\documentclass[12pt,a4paper]{amsart}

\usepackage{amssymb,amsfonts}
\usepackage[margin=2.9cm]{geometry}
\usepackage{subcaption}
\usepackage{graphics}
\usepackage{times}
\usepackage{amsmath}
\usepackage{graphicx} 
\usepackage{caption}

\usepackage{multicol}
\usepackage{mathtools}
\makeatletter

\newcommand*{\transp}[2][-3mu]{\ensuremath{\mskip1mu\prescript{\smash{\mathrm t\mkern#1}}{}{\mathstrut#2}}}

\newcommand\Label[1]{&\refstepcounter{equation}(\theequation)\ltx@label{#1}&}
\newcommand*\diff{\mathop{}\!\mathrm{d}}

\makeatother

\theoremstyle{plain}

\DeclareMathOperator{\WF}{WF}

\DeclareMathOperator{\rank}{rank}
\DeclareMathOperator{\sgn}{sgn}

\theoremstyle{plain}
\newtheorem*{cl*}{Claim}
\newtheorem{cl}{Claim}
\newtheorem*{pp*}{Proposition}	
\newtheorem{pp}{Proposition}
\newtheorem{lm}{Lemma}

\newtheorem{corollary}{Corollary}
\newtheorem{thm}{Theorem}

\theoremstyle{definition}
\newtheorem*{remark*}{Remark}
\newtheorem{example}{Example}
\newtheorem{df}{Definition}

\theoremstyle{definition}
\newtheorem*{eg*}{Example}
\begin{document}
\title[Recovery of singularities for the weighted cone transform]{Recovery of Singularities for the weighted cone transform appearing in the Compton camera imaging}
\author[Y. Zhang]{Yang Zhang}
\address{Purdue University \\ Department of Mathematics}
\email{zhan1891@purdue.edu}
\thanks{Partly supported by NSF Grant DMS-1600327}
\begin{abstract}
We study the weighted cone transform $I_\kappa$ of distributions with compact support in a domain $M $ of $\mathbb{R}^3$, over cone surfaces whose vertexes are located on a smooth surface away from $M$ and opening angles are limited to an open interval of $(0,\pi/2)$.
We show that 
when the weight function has compact support and satisfies certain nonvanishing assumptions,
the normal operator $I^*_\kappa I_\kappa$ is an elliptic $\Psi$DO at the accessible singularities. 
Then the accessible singularities are stably recoverable from local data.
We prove a microlocal stability estimate for $I_\kappa$.
Moreover,  we show the same analysis can be applied to the restricted cone transform.

\end{abstract}
\maketitle 
\section{Introduction}
Let $c(u,\beta,\phi)$ be a circular cone in $\mathbb{R}^3$ with vertex $u$, central axis $\beta$, and opening angle $\phi$, as shown in Figure \ref{conesfig}.
We study the weighted cone transform 
\[
I_\kappa f (u,\beta,\phi) =  \int_{c(u,\beta,\phi)} \kappa  f \diff S, \quad u \in \mathcal{S}, \beta \in S^2, \phi \in (\epsilon, \pi/2 - \epsilon)
\]
of distributions supported in a domain $M$ in $\mathbb{R}^3$ over cones of which the vertexes are restricted to a smooth surface $\mathcal{S}$, where $\kappa$ is a smooth weight and $S$ is the Euclidean measure on the cone. 
The goal of this work is to study the microlocal invertibility of this transform.

The cone transform arises in Compton camera imaging dating back to \cite{W.Todd1974,Everett1977,Singh1983}. 
A Compton camera is composed of two detectors: a scatter and an absorber. Both detectors are position and energy sensitive.
When incoming gamma photons hit the camera, they have Compton scattering at various angles in the first detector and are completely absorbed in the second one. Photons can be traced back to the surface of cones. 

\begin{figure}[h]
    \centering
    \includegraphics[height=0.25\textwidth]{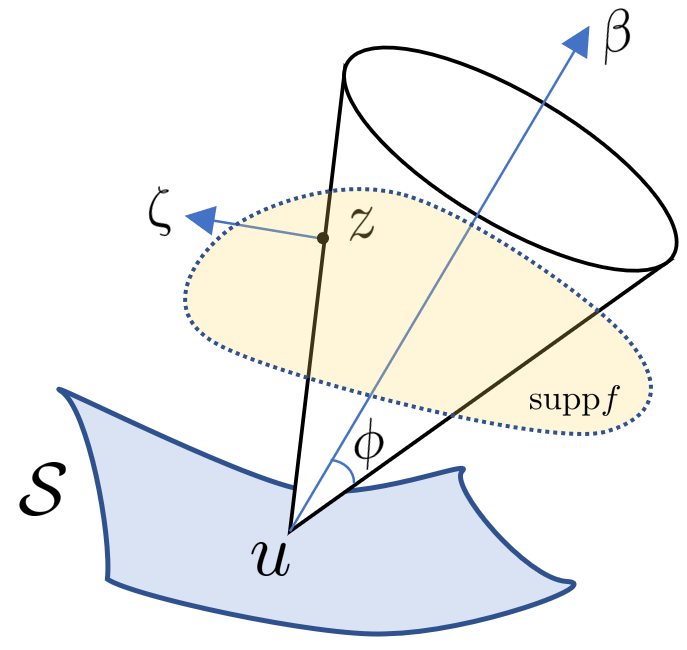}
    \caption{The cone $c(u,\beta,\phi)$ with $(z, \zeta)$ conormal to it }
\end{figure}\label{conesfig}

From the setting of Compton camera, there are three important points about the transform to note here. 
Firstly, it is natural and unavoidable to have a weight function $\kappa$.  
Since the integral over the cone is a superposition of the line integrals, even without attenuation, there should be weight $a/|z-u|$ with $a$ depending on the cone. 
Secondly, the probability of scattering angles (i.e. the opening angle $\phi$) governed by the Klein-Nishina distribution excludes angles that are too close to $0$ or $\pi/2$. 
For example, in \cite{Singh1983} the scattering angles range from $5^\circ$ to $75^\circ$ at $511$ keV incident energy and one obtains $5^\circ$ angular resolution by requiring the energy resolution to be 3.8-18.9 keV. 
On the other hand, when the scattering angle equals zero, the cone transform reduces to the weighted X-ray transform over a set of rays and when the scattering angle equals $\pi/2$, it reduces to the weighted Radon transform. In both cases the inversion are easier.
The limits of scattering angles can be modeled by choosing $\kappa$ supported in such intervals w.r.t. $\phi$. 
Thirdly, the detectors are located outside the region of interest $M$ so it is reasonable to require the vertexes are restricted to a smooth surface $\mathcal{S}$ that does not intersect $M$.

A lot of work has been done on the inversion of the cone transform and some of them are  \cite{Basko1998,Parra2000,Tomitani2002a,Smith2005,Maxim2009,Gouia-Zarrad2014a,Terzioglu2015,Kuchment2016,Jung2015,Schiefeneder2016,Moon2016,Terzioglu2017,Moon2017a,Terzioglu2018,Terzioglu2019}. 
For a more complete and detailed list of previous works, see \cite{Terzioglu2018}. 
Some of the inversion formula are for special geometries, or consider the opening angle $\phi \in (0,\pi)$, or constant weight, or use transform over cones whose vertexes can be everywhere.
Most recently, \cite{Terzioglu2019} considers a polynomial weight function and computes the normal operator of the cone transform over all cones whose vertexes are in $\mathbb{R}^n$ to show it is a $\Psi$DO. 
It gives certain integral formula for the amplitude of the normal operator but it is not so clear how to resolve the singularities in the denominator of the integrand to obtain a smooth amplitude.

For the cone transform that we define, it is harder to analyze the Schwartz kernel of the normal operator and the normal operator may fail to be a $\Psi$DO in some microlocal region, if Tuy's condition is not satisfied. 
Instead, we use the clean intersection calculus of FIOs in \cite{Duistermaat1994,Hoermander2009} to show under which conditions the microlocalized normal operator is a $\Psi$DO of order $-2$ and it is elliptic with certain nonvanishing assumptions of the weight, see Proposition \ref{I8I} and \ref{mI8I}.
This approach was proposed by Guillemin in \cite{Guillemin1985} for the generalized Radon transform . 
One difficulty to apply H{\"o}rmander's clean intersection composition to our transform is that the composition fails to be proper. 
For this purpose, we modify the clean composition theorem slightly by assuming the microsupport of composed FIO is conically compact instead of the properness condition following Andr{\'a}s Vasy's suggestion.
Our first result describes the singularities of $f$ that we can recover from the data $I_\kappa f$ in a stable way. 
We note that this is the intrinsic property of the transform itself no matter what inversion algorithm is used. More specifically, let 
\[
\mathcal{C}(z_0,\zeta^0) = \{(u,\beta,\phi)| (z_0 - u) \cdot \beta  - |z_0 - u|\cos \phi = 0, \ \zeta^0 \cdot (z_0 - u) = 0  \}
\] be the set of all cones that are conormal to fixed a $(z_0, \zeta^0) \in T^*M$, of which the dimension equals $2$. 
We can only expect to recover singularities conormal to the cones of which the weight is nonvanishing there.
The definition of accessible singularities and Tuy's condition can be found in Section 2.  
\begin{thm}
    Suppose $(z_0, \zeta^0)$ is accessible. If $\kappa(u_0, \beta_0, \phi_0, z_0) \neq 0$ for some $(u_0, \beta_0, \phi_0) \in  \mathcal{C}(z_0,\zeta^0)$, then $(z_0, \zeta^0)$ is recoverable. In other words, if $I_\kappa f$ is smooth near $(u_0, \beta_0, \phi_0)$, then $f$ is smooth near $(z_0, \zeta^0)$. 
\end{thm}
This theorem shows that accessible singularities are recoverable with some nonvanishing assumptions on $\kappa$.
The recovery comes from the fact that with the assumptions the normal operator in certain microlocal region is an elliptic $\Psi$DO and therefore it is microlocally invertible.
We also shows the mapping properties of $I_\kappa, I_\kappa^*$ and a microlocal stability estimate when all singularities are accessible. 
\begin{thm}
    Suppose $\mathcal{S}$ satisfy Tuy's condition w.r.t. $M$. 
    Then for any $s\in \mathbb{R}$, we have $I_\kappa:  H^s_{loc}(M) \rightarrow H^{s+1}_{loc}(\mathcal{M}), \quad I_\kappa^*: H^s_{loc}(\mathcal{M}) \rightarrow H^{s+1}_{loc}({M})$ are continuous.
    For $f \in H^s(M), l \in \mathbb{R}$, there exists constant $C_1, C_2, C_{s,l}$ such that
    \[
    C_1 \| f\|_{H^s(M)} - C_{s,l}\| f\|_{H^l(M)} \leq \|I_\kappa f\|_{H^{s+1}(\mathcal{M})} \leq C_2 \|f\|_{H^s(M)}.
    \]
\end{thm}
For surfaces satisfying Tuy's condition, see Example \ref{sphereeg}. We should mention that assuming $\mathcal{S}, \kappa$ are analytic, one might be able to show the injectivity of $I_\kappa$ by applying the analytic microlocal analysis, see \cite{Frigyik2007}.
Once we have the transform $I_\kappa$ is injective on some closed subspace of $H^s(M)$, 
we can show the microlocal stability actually implies the stability estimate, i.e. we do not have the term $C_{s,l}\| f\|_{H^l(M)}$ in above inequality.

This work is structured as follows. 
Section \ref{preliminaries} introduces the notations and definitions. 
Section \ref{sectionFIO} proves that $I_\kappa$ is an FIO and show in which cases the projection $\pi_\mathcal{M}$ is an injective immersion. 
Section \ref{sectionCIC} presents a slightly modified version of the clean composition theorem.
Section \ref{sectionnormal} shows with certain positive assumptions on the weight $\kappa$, 
the normal operator is an elliptic $\Psi$DO at accessible singularities, which implies Theorem 1.
Section \ref{sectionproof2} contains the proof of Theorem 2.
Section \ref{sectionrestricted} studies the weighted cone transform over cones whose vertexes are restricted to a smooth curve and the opening angle is fixed.

\subsection*{Acknowledgments}
The author would like to thank Plamen Stefanov for suggesting this problem and for numerous helpful discussion with him throughout this project, and to thank Andr{\'a}s Vasy for helpful suggestions on the clean intersection calculus part.

\section{Preliminaries}\label{preliminaries}
Here we introduce some notations to be used in the following sections.
As mentioned before, let $M$ be an open domain in $\mathbb{R}^3$ and we always assume $f\in \mathcal{E}'(M)$.
Suppose $\mathcal{S}$ is a smooth regular surface without boundary where the vertexes of cones are located. 
Let $0 \leq \epsilon<\pi/4$. 
Let $\mathcal{M} = \mathcal{S} \times S^2 \times (\epsilon,{\pi}/{2}-\epsilon)$ be the family of cones that we consider. 
Notice $\mathcal{M}$ is a smooth manifold without boundary.

Since both $u, \beta$ are in smooth surface, we consider them in local coordinates. 
Suppose $\mathcal{S}$ has the regular parameterization locally given by $u = u(v^1,v^2)$. 
Let $J_1 = (\frac{\partial u}{\partial v^1}, \frac{\partial u}{\partial v^2}) \equiv (r_1, r_2)$ be the Jacobian matrix.
Notice $r_1, r_2$ form a basis for the Tangent space at $u$ of $\mathcal{S}$.
Suppose $\beta \in S^2$ is locally parameterized by $\beta = (\sin\theta \cos \psi, \sin \theta \sin \psi, \cos \theta)$ for $\theta \in (0, \pi)$. 
Let $J_2 = (\frac{\partial \beta}{\partial \theta}, \frac{\partial \beta}{\partial \psi}) \equiv (\beta_1, \sin \theta \beta_2)$, where $\beta_1 = (\cos\theta \cos \psi, \cos \theta \sin \psi, -\sin \theta)$ and $\beta_2 = (-\sin \psi, \cos \psi,0)$.
Notice $\beta, \beta_1, \beta_2$ form an orthonormal basis in $\mathbb{R}^3$.

\begin{df}
We say $(z,\zeta) \in T^*M$ is \textit{accessible} by $\mathcal{S}$ if the hyperplane conormal to $(z,\zeta)$ has a non-tangential intersection with $\mathcal{S}$. 
\end{df}
If we donate the set of all $(z,\zeta)$ in $T^*M$ that are accessible by $D$, then $D$ is an open set. Observe that $(z,\zeta)$ is accessible implies that exists a cone that it is conormal to, and for any qualified cone $(u, \beta, \phi)$, the covector $\zeta$ cannot be perpendicular to the tangent space at $u$ in $\mathcal{S}$.
\begin{df}
If every $(z,\zeta) \in T^*M$ is accessible, then we say the surface $\mathcal{S}$ satisfying the \textit{Tuy's condition} with respect to $M$.
\end{df}
This coincides with the definition in \cite{Terzioglu2017} that Tuy's condition is satisfied if any hyperplane intersecting $M$ has a non-tangential intersection with $\mathcal{S}$.

\section{$I_\kappa$ as an FIO}\label{sectionFIO}
The weighted cone transform $I_\kappa$ can be written as 
\begin{align*}
I_\kappa f(u,\beta,\phi) =  \int_{\mathbb{R}^3} \kappa(u,\beta,\phi,z) \delta((z-u) \cdot \beta - |z-u| \cos \phi) f(z) 
\diff z, \\
\quad u \in \mathcal{S}, \beta \in S^2, \phi \in (\epsilon, \pi/2 -\epsilon)
\end{align*}
where the distribution $\delta$ has a nonzero factor but we can regard the factor as a part of the weight.
\begin{pp}\label{prop_CI}
    The weighted cone transform $I_\kappa$ is an FIO of order $-{3}/{2}$ associated with the canonical relation
    \begin{align*}
    C_I
    & = \{ (u, \beta, \phi,
    \underbrace{ \transp{J_1} \zeta \vphantom{\lambda |z-u|\sin\phi}}_{\hat{u}},
    \underbrace{\lambda \transp{J_2}  (z-u) \vphantom{\lambda |z-u|\sin\phi}}_{\hat{\beta}},
    \underbrace{\lambda |z-u|\sin\phi  \vphantom{\lambda |z-u|\sin\phi}}_{\hat{\phi}},
    z,
    \zeta 
    ), \ 
    \varphi=0,
    \lambda \neq 0
    \},
    \end{align*}
    where 
    $\varphi(u,\beta,\phi,z) = (z-u) \cdot \beta - |z-u| \cos \phi$
    and $\zeta = -\lambda(\beta - \frac{z-u}{|z-u|}\cos\phi)$; 
    the vertex $u = u(v^1,v^2)$ locally and $J_1 = \frac{\partial u}{\partial(v^1,v^2)}$ is the Jacobian matrix; 
    the unit vector $\beta=\beta(\theta,\psi)$ is locally parameterized in the spherical coordinates 
    and $J_2 = \frac{\partial \beta}{\partial(\theta,\psi)}$ is the Jacobian matrix.
\end{pp}
\begin{proof}
    We rewrite $I_\kappa$ as the oscillatory integral
    \[
    I_\kappa f (u,\beta,\phi) =  (2\pi)^{-1} \int_\mathbb{R^3} \int_\mathbb{R} e^{i  \lambda \varphi(u,\beta,\phi,z)} \kappa(u,\beta,\phi,z) f(z) \diff \lambda \diff z,
    \]
    where $\varphi(u,\beta,\phi,z)$ is defined as above.
    Notice that $\diff \varphi \neq 0$.
    Its Schwartz kernel is a Lagrangian distribution conormal to the characteristics manifold
    \[
    Z = \{(u,\beta,\phi,z) \in \mathcal{M}\times M, \ 
    \varphi = 0 \}.
    \]
    Then the conormal bundle is
    \begin{align*}
    N^*Z 
    & = \{ (u,\beta,\phi,z,
    \lambda \diff_u \varphi, \lambda \diff_\beta \varphi, \lambda \diff_\phi \varphi, \lambda \diff_z \varphi
    ), \ 
    \varphi = 0\}.
    \end{align*}
    where we abuse the notation $\diff_u$, $\diff_\beta$ to denote the differential w.r.t. the parameterization of $u$, $\beta$ and we have
    \begin{align}
    & \diff_z \varphi =  \beta - \frac{z-u}{|z-u|}\cos\phi, \quad \diff_u \varphi = - \transp{J_1} \diff_z \varphi, \label{varphieq1}\\
    &\diff_\beta \varphi = \transp{J_2} (z-u), \quad
    \diff_\phi \varphi =  |z-u| \sin \phi. \label{varphieq2}
    \end{align}
    Let 
    \[
    \Lambda  = \{ (u,\beta,\phi,z,
    \lambda \diff_u \varphi, \lambda \diff_\beta \varphi, \lambda \diff_\phi \varphi, \lambda \diff_z \varphi
    ), \ 
    \varphi = 0, \lambda \neq 0\}.
    \]
    One can show this is a closed conic Lagrangian submanifold of $T^*(\mathcal{M}\times M) \setminus 0$. 
    
    Additionally, since we always have $M$ is away from $\mathcal{S}$, 
    from (\ref{varphieq1})(\ref{varphieq2}) we get that
    \[
    d_z \varphi =0 \Longleftrightarrow
    d_u \varphi =d_\beta \varphi = d_\phi \varphi=0 \Longleftrightarrow
    \sin \phi = 0,
    \quad \text{with } \varphi =0.
    \]
    Since $\phi \in (\epsilon,\frac{\pi}{2}-\epsilon)$, therefore the Lagrangian satisfies 
    \[
    \Lambda \subset (T^*\mathcal{M} \setminus 0)\times (T^*M \setminus 0).
    \]
    Then the canonical relation $C_I$ is given by the twisted Lagrangian. The order is given by $0 + \frac{1}{2} \times 1 - \frac{1}{4} \times (3 + 5) = -\frac{3}{2}$ by \cite[Definition 3.2.2]{Hoermander2003}.
\end{proof}
Note that we have $\dim T^*\mathcal{M} = 10$, $\dim T^*M = 6$,  and $\dim C_I = 8$. 
Let $\pi_\mathcal{M}$, $\pi_{M}$ be the natural projection of $C_I$ to $T^*\mathcal{M},T^*M$ respectively.
Notice in cone transform, neither of these projections can be local diffeomorphism.
The following proposition describes the mapping properties of them.

\begin{pp}\label{propertyofC}
    Recall $D$ is the set of all accessible covectors $(z,\zeta)$ in $T^*M $. 
    Notice $\pi_{\mathcal{M}} \pi^{-1}_M(D) =\{ (u,\beta,\phi,\hat{u},\hat{\beta},\hat{\phi}) \in C_I(z,\zeta)|(z,\zeta) \in D\}$.
    Let $H_{(z,\zeta)}$ be the hyperplane conormal to $(z,\zeta)$.
    We have the following statements.
    \begin{itemize}
        \item[(a)] For each $(z,\zeta) \in D$, the set $C_I(z,\zeta)$ is a surface that can be parametrized by $(u,\phi) \in U_{(z,\zeta)} \times (\epsilon,{\pi}/{2}-\epsilon)$,
        where $U_{(z,\zeta)} = \mathcal{S} \cap H_{(z,\zeta)}$, see Figure \ref{exampleofu} as examples.
        \item[(b)] For each
        $(u,\beta,\phi,\hat{u},\hat{\beta},\hat{\phi}) 
        \in \pi_{\mathcal{M}} \pi^{-1}_M(D) $, 
        there is one unique solution $(z,\zeta)$ for the equation $C_I(z,\zeta)=(u,\beta,\phi,\hat{u},\hat{\beta},\hat{\phi})$, which is given by (\ref{solution_z}) and below.
        \item[(c)] The projection $\pi_\mathcal{M}$ restricted to $\pi^{-1}_\mathcal{M}(D)$ is an injective immersion. 
        In particular, if $\mathcal{S}$ satisfies Tuy's condition, then $\pi_\mathcal{M}$ itself is an injective immersion. 
    \end{itemize}
\end{pp} 
Before the proof, the following are two examples to illustrate its first statement.
   \begin{example}
    Suppose $\mathcal{S}$ is a plane. W.L.O.G., assume $\mathcal{S} = \{u^3 = 0\}$.
    Let $z = (z^1,z^2, z^3), \zeta = (\zeta_1,\zeta_2,\zeta_3)$. 
    If $\zeta_1 = \zeta_2 = 0$, then there are no cones that are conormal to $(z,\zeta)$.
    Here we assume $\zeta_2 \neq 0$.
    We solve
    \[
    (z-u) \cdot\zeta = \zeta_1(z^1 - v^1) + \zeta_2(z^2 - v^2) + \zeta_3 z^3 = 0
    \]
    to get $v^2= \frac{1}{\zeta_2}(z\cdot \zeta - v^1 \zeta_1)$. 
    Thus, for $(u,\beta,\phi, \hat{u}, \hat{\beta},\hat{\phi}) \in C_I(z,\zeta)$, we have
    \begin{align*}
    &u= (v^1,\frac{1}{\zeta_2}(z\cdot \zeta - v^1 \zeta_1),0 ),
    &\beta = \cos \phi \frac{z-u}{|z-u|} - \sin \phi \frac{\zeta}{|\zeta|},\\
    &\phi, v^1 \text{ arbitrary,}
    &\lambda = \frac{1}{\sin\phi} |\zeta|.
    \end{align*}
    \begin{figure}[h]
        \centering
        \begin{subfigure}{0.5\textwidth}
            \centering
            \includegraphics[height=0.5\textwidth]{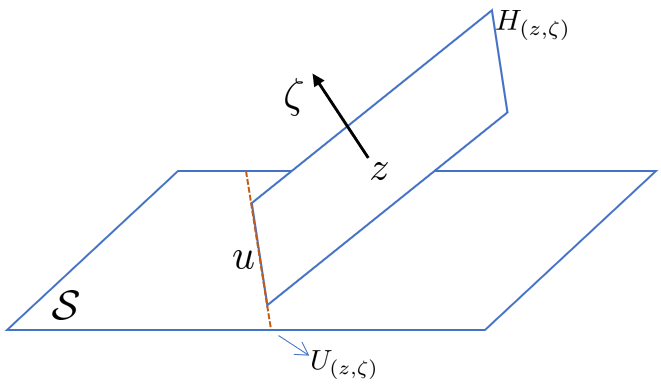}
        \end{subfigure}%
        \begin{subfigure}{0.5\textwidth}
            \centering
            \includegraphics[height=0.45\textwidth]{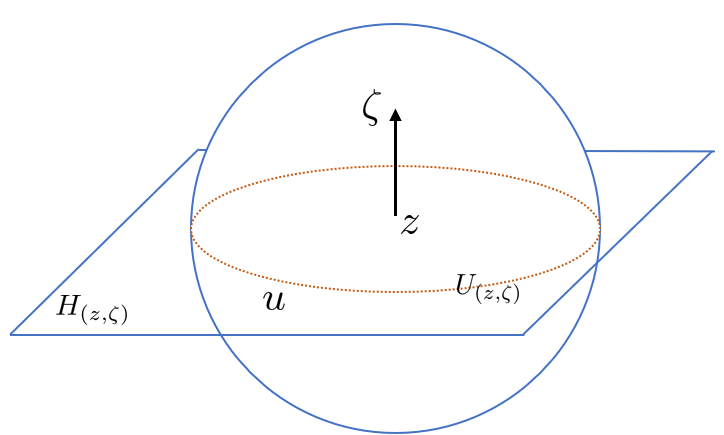}
        \end{subfigure}
        \caption{Example 1 and 2}\label{exampleofu}
    \end{figure}
\end{example}
\begin{example}\label{sphereeg}
    Suppose $\mathcal{S}$ is a unit sphere. 
    Note in this case  $\mathcal{S}$ satisfies Tuy's condition w.r.t. any inside domain away from it.
    It can be covered by six coordinates charts.
    We consider the special case that $z = (0,0,0)$ and $\zeta = (0,0,1)$.
    Then \[
    U_{(z,\zeta)} = \{(v^1,v^2,0), \ v^1 v^1 + v^2 v^2 = 1 \}.
    \]
    In a small neighborhood of $U_{(z,\zeta)}$, 
    the vertex $u$ can be parameterized by one of the following
    \[(v^1, \sqrt{1 - (v^1)^2 - (v^3)^2}, v^3), \quad (v^1, -\sqrt{1 - (v^1)^2 - (v^3)^2 }, v^3),\]
    \[(\sqrt{1 - (v^2 )^2 - (v^3 )^2},v^2, v^3), \quad (-\sqrt{1 - (v^2)^2 - (v^3 )^2},v^2, v^3).\]
    Notice that $\partial_{v^3} u \cdot \zeta \neq 0$ and therefore $U_{(z,\zeta)}$ can be parameterized by $v^1$ or $v^2$ from the proof above. This can also be seen from $ U_{(z,\zeta)} = \{(v^1,v^2,0), \ v^1 v^1 + v^2 v^2 = 1 \}$ itself.
\end{example}
\begin{proof}[Proof of Proposition \ref{InjectiveImmersion}]
    For (a), given $(z,\zeta)$, we are going to find out all possible solutions of $(u,\beta,\phi, \lambda)$ from the canonical relation in Proposition \ref{prop_CI}. 
    We have some freedom to choose $u \in \mathcal{S}$,
    but with $\zeta = - \lambda(\beta - \frac{z-u}{|z-u|} \cos\phi  )$, 
    the vector $(z-u)$  must be conormal to $\zeta$.
    This coincides with the fact that the singularity $(z,\zeta)$ can only be possibly detected by the cones that it is conormal to.
    In other words, the possible choice of $u$ is the set $ U_{(z,\zeta)}$. 
    Indeed, the vertex $u$ should satisfy the equation $g(z,\zeta,u) = 0$, where
    $g= (z-u) \cdot \zeta$. The Jacobian matrix is listed in the following,
    \begin{figure}[h]
        \includegraphics[height=0.12\textwidth]{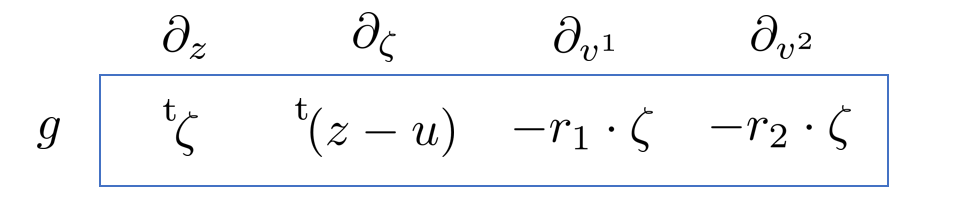}
    \end{figure}\\
    with $r_1 ={\partial_{v^1}} u$, $r_2 ={\partial_{v^2}} u$.
    Here $r_1, r_2$ form a basis of the tangent space $T_u \mathcal{S}$.
    Since $(z,\zeta)$ is accessible,
    the inner products $r_1 \cdot \zeta $ and $r_2 \cdot \zeta $ cannot vanish at the same time; 
    otherwise, the covector $\zeta$ is normal to $\mathcal{S}$ at $u$. 
    We simply assume $r_2 \cdot \zeta \neq 0 $ in a small neighborhood of fixed $u_0 \in U_{(z,\zeta)}$. 
    In this neighborhood, the derivative $\partial_{v^2} g \neq 0$.
    Applying implicit function theorem, we get $v^2$ is a smooth function of $z,\zeta,v^1$ near $u_0$. 
    Locally $U_{(z,\zeta)}$ can be parameterized by $v^1$.
    
    Next, with $u_0$ given, 
    by choosing $\phi$, the axis $\beta$ can be determined by $\beta = \cos \phi \frac{z-u}{|z-u|} - \sin \phi \frac{\zeta}{|\zeta|}$ and $\lambda = \frac{1}{\sin\phi} |\zeta|$.
    From Proposition \ref{prop_CI},
    \[
    \hat{u} = \transp{J}_1(v_1,v_2) \zeta, \quad
    \hat{\beta} = \lambda \transp{J}_2(\beta) (z-u), \quad
    \hat{\phi} = \lambda |z-u|\sin \phi.
    \]
    These are all smooth functions of $z, \zeta, u, \beta, \phi, \lambda$ and therefore smooth functions of $z, \zeta, v^1, \phi$.
    The map $(v^1, \phi) \mapsto (u,\beta,\phi,\hat{u},\hat{\beta},\hat{\phi}) \in C_I(z,\zeta)$ is an immersion. Thus, 
    $(v^1,\phi)$ is a local parameterization of $C_I(z,\zeta)$. This proves statement (a).

    For (b), to recover $(z,\zeta)$ from given $(u,\beta,\phi,\hat{u},\hat{\beta},\hat{\phi})$, we are solving the following system of equations.
    \begin{align*}
    \Label{one} & (z-u) \cdot \beta - |z-u|\cos \phi = 0  & \Label{two} & \lambda((z-u) \cdot \beta_1, \sin \theta (z-u)\cdot \beta_2) = \hat{\beta} \\
    \Label{three} & \lambda |z-u| \sin \phi = \hat{\phi} & \Label{four} & \lambda (\beta - \frac{(z-u)}{|z-u|}\cos\phi) = \hat{u} 
    \end{align*}
    Recall $(z-u)/|z-u| = m \in S^2$. 
    We have $\hat{\phi}$ is always nonzero
    and $\hat{u}$ is nonzero with the assumption that $(u,\beta,\phi,\hat{u},\hat{\beta},\hat{\phi}) \in \pi_{\mathcal{M}} \pi^{-1}_M(D)$.
    Indeed, $\hat{u}$ vanishes if and only if $\transp{J_1} (\beta - m\cos\phi) = (r_1 \cdot (\beta - m\cos\phi), r_2 \cdot (\beta - m\cos\phi))= 0 \implies (r_1 \cdot \zeta, r_2 \cdot \zeta) =0$.
    As we stated before, this contradicts with the assumption. 
    
    We are solving the system. Divided by $|z-u|$ and $\lambda|z-u|$ respectively, 
    the equations (\ref{one}) and (\ref{two}) give us the projection of $m$ along $\beta$, $\beta_1$, and $\sin \theta \beta_2$. 
    These vectors form a orthogonal basis in $\mathbb{R}^3$ and we can get $m$ from the projection.
    Plugging back $m$ into equation (\ref{four}), we can solve $\lambda$. 
    Thus, when $\theta \neq 0$ or $\pi$
    \begin{align}\label{solution_z}
    &z = u + \frac{\hat{\phi}}{\lambda \sin\phi} m 
    \quad &  \zeta = - \lambda(\beta - \cos\phi m ),
    \end{align}
    where
    \[
    m =\cos \phi \beta + \frac{\sin \phi}{\hat{\phi}}(\hat{\beta}_1 \beta_1 + \frac{1}{\sin \theta}\hat{\beta}_2 \beta_2) 
    \quad \lambda = \frac{ |\hat{u}|\text{\(\sgn\)}\hat{\phi}}{|\transp{J_1} (\beta - m \cos \phi)|}.
    \]
    When $\theta =0$ or $pi$, this argument still holds if we choose another regular parameterization of $\beta$ there.
    
    For (c), first we prove $\pi_\mathcal{M}: (u,\beta,\phi,\hat{u},\hat{\beta},\hat{\phi},z,\zeta) \mapsto  (u,\beta,\phi,\hat{u},\hat{\beta},\hat{\phi})$ is an immersion.
    It suffices to prove that $\text{\(\rank\)} (d \pi_\mathcal{M}) = \dim \mathcal{C}_I = 8$.
    The canonical relation $C_I$ has a parametrization $(u,\beta,z,\hat{\phi})$, which implies $\dim \mathcal{C}_I = 8$.
    Indeed, we can solve $\phi, \lambda$ directly from these parameters and therefore $C_I$ can be represented by them.
    We list the Jacobian matrix with respect to these variables in the following. 
    Here we write $J_1 = (r_1, r_2)$ 
    with $r_1 ={\partial_{v^1}} u$ and $r_2 ={\partial_{v^2}} u$,
    which form a basis of $T_u^*\mathcal{S}$;
    the matrix $J_2 = (\beta_1, \beta_2)$, 
    where $\beta_1 ={\partial_\theta} \beta$ and $\beta_2 = {\partial_\psi} \beta$ are unit vectors orthogonal to $\beta$.
    The matrix $J_3$ has $\text{\(\rank\)} = 3$; otherwise we have for $j = 1,2$,
    \begin{align*}
    (r_j - (r_j \cdot m)m)\cdot \beta = 0
    \Longrightarrow  (r_j \cdot \beta - m \cos \phi) =0,
    \end{align*}
    which implies $\beta - m \cos \phi$ is normal to $\mathcal{S}$ at $u$. By the same arguments as before, this contradicts with
    the assumption.
    Observe that $\text{\(\rank\)}J_1 = \text{\(\rank\)}J_2 = 2$. This proves that  $\text{\(\rank\)}( d \pi_\mathcal{M}) = 8$.
    \begin{figure}[h]
        \centering
        \begin{subfigure}{0.5\textwidth}
            \centering
            \includegraphics[height=0.6\textwidth]{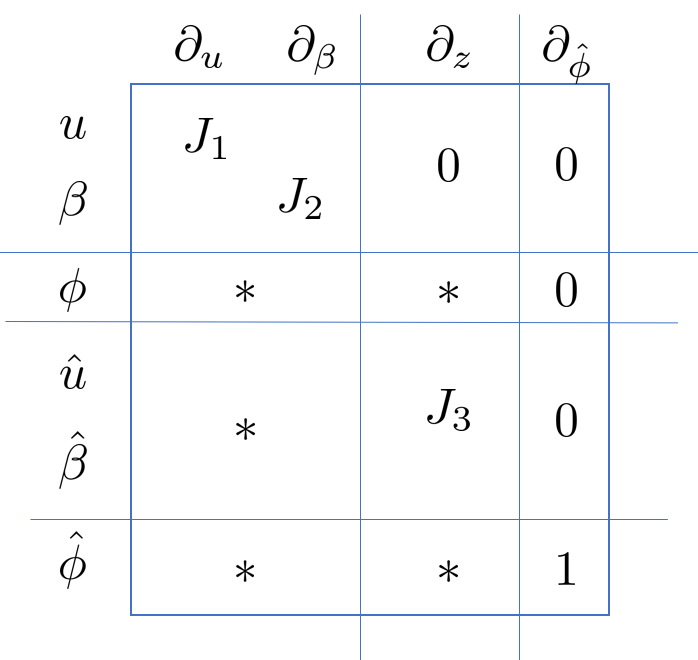}
        \end{subfigure}%
        \begin{subfigure}{0.5\textwidth}
            \centering
            \includegraphics[height=0.5\textwidth]{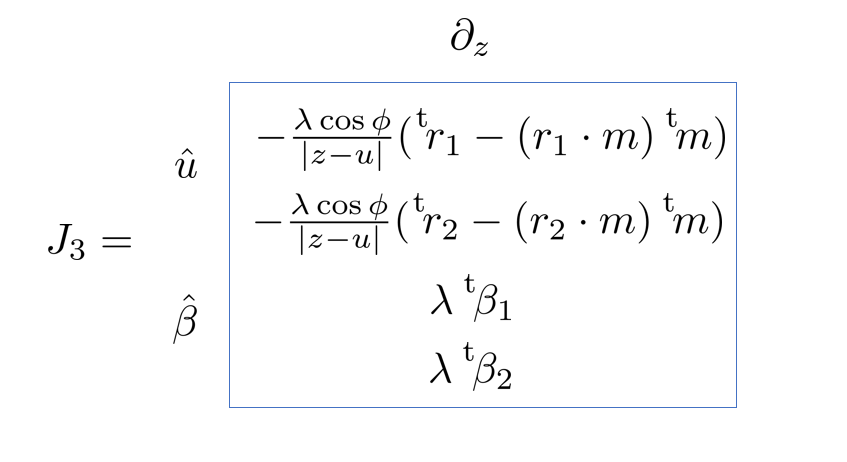}
        \end{subfigure}
    \end{figure}
    
    The injectivity of $\pi_\mathcal{M}$ comes from the statement (b). In particular, if $\mathcal{S}$ satisfies Tuy's condition, 
    then $R = \{(u,\beta,\phi,\hat{u},\hat{\beta},\hat{\phi}) \in C_I(z,\zeta)|\forall (z,\zeta) \in T^*M \} = \pi_\mathcal{M}(C_I)$.
    And for every $(u,\beta,\phi,\hat{u},\hat{\beta},\hat{\phi}) \in \pi_\mathcal{M}(C_I)$, 
    there is a unique $(z,\zeta)$ such that 
    \[
    \pi_\mathcal{M}(u,\beta,\phi,\hat{u},\hat{\beta},\hat{\phi},z,\zeta) = (u,\beta,\phi,\hat{u},\hat{\beta},\hat{\phi}).
    \]   
\end{proof}
\section{Clean Intersection Calculus}\label{sectionCIC}
In this section, we present the clean intersection calculus stated in \cite[Theorem 25.2.3]{Hoermander2009}.
We modify it slightly to better suit our problem.
\begin{pp}\label{CIC}
    Let $X$, $Y$, $Z$ be three $C^\infty$ manifolds.
    Suppose $C_1$ from $(T^*Y \setminus 0)$ to $ (T^*X \setminus 0)$ is a $C^\infty$ homogeneous canonical relation 
    closed in $ T^*(X \times Y) \setminus 0$  and $C_2$ another from $(T^*Z \setminus 0)$ to $ (T^*Y \setminus 0)$. 
    Let \[
    A_1 \in I^{m_1} (X \times Y, C_1'), \quad A_2 \in I^{m_2} (Y \times Z, C_2')
    \] be properly supported FIOs with principal symbol $\alpha_1$ and $\alpha_2$ respectively.
    Assume we have 
    \begin{itemize}
        \item[(A1)] $C_1 \times C_2$ intersects $ T^*X \times \Delta(T^*Y) \times T^*Z$ cleanly with excess $e$. We call the intersection $\hat{C}$,
        \item[(A2)] 
        there exist conically compact subsets $K_1, K_2$ of $C_1, C_2$ respectively, such that the microsupport of $A_1, A_2$ is in some open set of $K_1, K_2$ respectively.
        \item[(A3)] the inverse image $C_\gamma$ under the projection $ \hat{C} \rightarrow T^*(X \times Z)\setminus 0$ of any $\gamma \in C \equiv C_1 \circ C_2$ is connected.
    \end{itemize}
    Then there exists an open set $\mathcal{O}$ in $ T^*X \times T^*Z$ such that $\mathcal{O} \cap C$ is a conic Lagrangian submanifold and 
    \[
    A_1 A_2 \in I^{m_1 + m_2 + e/2} (X \times Z, (\mathcal{O} \cap C)')
    \]
    and for the principal symbol $\alpha$ of $A_1 A_2$ we have 
    \begin{align}\label{principalsymbol}
     \alpha = \int_{C_\gamma} \alpha_1 \times \alpha_2, 
    \end{align}   
    where $\alpha_1 \times \alpha_2$ is the density on $C_\gamma$ as is defined in \cite[Theorem 25.2.3]{Hoermander2009}.
\end{pp}
\begin{remark*}
    If we compare this proposition with Theorem 25.2.3, the difference is that H{\" o}rmander assumes the restricted projection $\Pi: \hat{C} \rightarrow T^*(X \times Z)\setminus 0$ is proper instead of condition (2) above.
    The properness of restricted $\Pi$ has the following implications in the original proof.
    First, since by \cite[Theorem 21.2.14]{MR2304165} the restricted $\Pi$ has constant rank, if $\Pi$ is proper and has connected fibers (i.e. the preimage of a single point is connected, that is condition (3)), then the image $C = \Pi(\hat{C})$ is an embedded submanifold, see Lemma \ref{constantrankclosed} and its remark.
    Second, since a continuous map is proper if and only if it is closed and has compact fibers (i.e. the preimage of a single point is compact), we have $C$ is a closed submanifold and $C_\gamma$ is compact. The later implies the integral 
    $\int_{C_\gamma} \alpha_1 \times \alpha_2$ is well-defined.
\end{remark*}
\begin{proof}[Proof of Proposition \ref{CIC}: part 1]
    We prove in the following that there exists an open set $\mathcal{O} \subset T^*X \times T^*Z$ such that $\mathcal{O} \cap C$ is an embedded submanifold of $T^*X \times T^*Z$.
    
    Let $\Pi$ be the projection 
    \[
    \Pi: T^*X \times \Delta(T^*Y) \times T^* Z \rightarrow T^*X \times T^*Z.
    \]
    Since $\Pi$ is an open map, we can choose $\mathcal{O}$ by choosing an open subset $\hat{\mathcal{O}}$ of $T^*X \times \Delta(T^*Y) \times T^* Z$ by condition (2). 
    More specifically, recall $\hat{C} = (C_1 \times C_2)\cap (T^*X \times \Delta(T^*Y) \times T^*Z)$ and we define 
    \begin{align*}
    \hat{K} &= (K_1 \times K_2)\cap (T^*X \times \Delta(T^*Y) \times T^*Z),\\
    \hat{W} &= (\text{\( \WF \)}(A_1) \times \text{\( \WF \)}(A_2))\cap (T^*X \times \Delta(T^*Y) \times T^*Z)
    \end{align*}
    Let $\hat{\mathcal{O}}$ be an open set such that $\hat{W} \subset \hat{\mathcal{O}} \cap \hat{C} \subset \hat{K} \subset \hat{C}$. Then we have $\mathcal{O} = \Pi(\hat{\mathcal{O}})$ is an open subset.
    
    We can show $\mathcal{O} \cap C$ is an embedded submanifold in two ways. 
    On the one hand, since $\Pi$ restricted on $\hat{K}$ is proper and $\hat{\mathcal{O}} \cap \hat{C} \subset \hat{K}$, 
    by the proof of Lemma \ref{constantrankclosed}, each point in $\mathcal{O} \cap C$ has a submanifold coordinate chart.
    On other other hand, we can prove it by the following claim. 
    \begin{cl}\label{Oclosed}
        The restricted projection $\Pi: \hat{\mathcal{O}} \cap \hat{C} \rightarrow \mathcal{O}$ is a closed map.
    \end{cl} 
    Also by Lemma \ref{constantrankclosed}, we have $\mathcal{O} \cap C$ is an embedded submanifold of $\mathcal{O}$ and therefore that of $T^*X \times T^*Z$.
    
    Then, by Claim \ref{waveset}, the wave front set of $A$ is contained in $\Pi(\hat{W})$ and thus is contained in $ \mathcal{O} \cap C$.
    Now we can define the Fourier integral distributions on the embedding Lagrangian submanifold $C \cap \mathcal{O}$ by 
    \cite[Lemma 25.1.2]{Hoermander2009} and its remark. Although $\mathcal{O} \cap C$ is not necessarily closed, by the remark in \cite[p.~147]{Hoermander1971}, we can require the symbols (the amplitude in any specific representation) vanishing outside a closed conic subset and have the same conclusions about the principal symbols.
\end{proof}
\begin{proof}[Proof of Proposition \ref{CIC}: part 2]
    We follow the proof of H{\"o}rmander and skip most of it here.
    One can first show $A_1 A_2$ can be written as a sum of an FIO associated with $\mathcal{O} \cap C$ and a smoothing operator, and then compute the principal symbol.
    The only difference is that we still need to verify that the integration in (\ref{principalsymbol}) is well-defined.
    We claim that the principal symbol $\alpha_1$ and $\alpha_2$ have conically compact support and therefore it is well defined. 
    Indeed, consider the local representation of $A_1, A_2$
    \begin{align*}
    & A_1(x,y) = (2\pi)^{-(n_X + N_Y + 2N_1)/4} \int e^{i\phi(x,y,\theta)} a_1(x,y,\theta) \diff \theta,\\
    & A_2(y,z) = (2\pi)^{-(n_X + N_Y + 2N_2)/4} \int e^{i\psi(x,y,\tau)} a_2(x,y,\tau) \diff \tau,
    \end{align*}
    where $\phi, \psi$ are non-degenerate phase functions and $a_1, a_2$ are amplitudes. For more details see \cite[Theorem 25.2.3]{Hoermander2009}.
    Then the principal symbols of $A_1, A_2$ are
    \begin{align}\label{principalsymbolformal}
    \alpha_1 = a_1(x,y,\theta) e^{\pi i/4 \sgn H_\phi} d_{C_1}^\frac{1}{2}, \quad  \alpha_2 = a_2(y,z,\tau) e^{\pi i/4 \sgn H_\psi} d_{C_2}^\frac{1}{2},  
    \end{align}    
    according to \cite[p.~14]{Hoermander2009}.
    Since $A_1, A_2$ have conically compact microsupport respectively, it suffices to show that the principal symbol vanishes outside the microsupport. 
    Indeed, for fixed $(x_0, y_0, \xi^0, \eta^0) \notin \text{\( \WF \)}(A_1)$ in the conical relation, 
    there exists a small conic neighborhood of $(x_0, y_0, \xi^0, \eta^0)$ 
    such that the local representation of $A_1$ above is smooth. 
    By \cite[Lemma 4.1]{Treves1980}, there is $a_{1,\infty} \in S^{-\infty}$ such that $a_1 = a_{1,\infty}$ on $\Sigma_\phi = \{(x,y,\theta), \ \phi'_\theta =0 \}$. 
    On the other hand, the principal symbol $\alpha_1$ defined in (\ref{principalsymbolformal}) only depends the amplitude restricted to the the manifold $\Sigma_\phi$. Therefore, it vanishes in this neighborhood.
    
\end{proof}
\begin{lm}\label{constantrankclosed}
    Let $f: X \rightarrow Y$ be a smooth map with constant rank. If $f$ is closed and the preimage $f^{-1}(y)$ of any $y \in f(X)$ is connected, then $f(X)$ is an embedded submanifold of $Y$.
\end{lm}
\begin{remark*}
    In particular, this lemma holds if we assume $f$ is proper instead of simply closed. A slightly different version of this lemma can be found in \cite{Guillemin2005}. It claims the proof can be found in Appendix C.3 in \cite{MR2304165}.
    The proof we write in the following is based on the outlines of H{\" o}rmander and we borrow most of it from an online answer that proves the case when $f$ is proper. 
\end{remark*}
\begin{remark*}
    It is necessary to assume that  $f$ is closed. 
    Consider projection $\pi$ from $\mathbb{R}^3$ to the $xy$ plane.
    Let $\nu_1$ be a smooth curve as is shown in Figure \ref{counterexamples} (a), whose image under projection is the figure $6$. Here $f$ is the projection $\pi$ restricted on $\nu_1$. 
    The preimage for each point in $\pi(\nu_1)$ under $f$ is connected but $f$ is not closed. The image $f(\nu_1)$ is not an embedded submanifold.
    \begin{figure}[h] 
        \centering
        \begin{subfigure}{0.5\textwidth}
            \centering
            \includegraphics[height=0.6\textwidth]{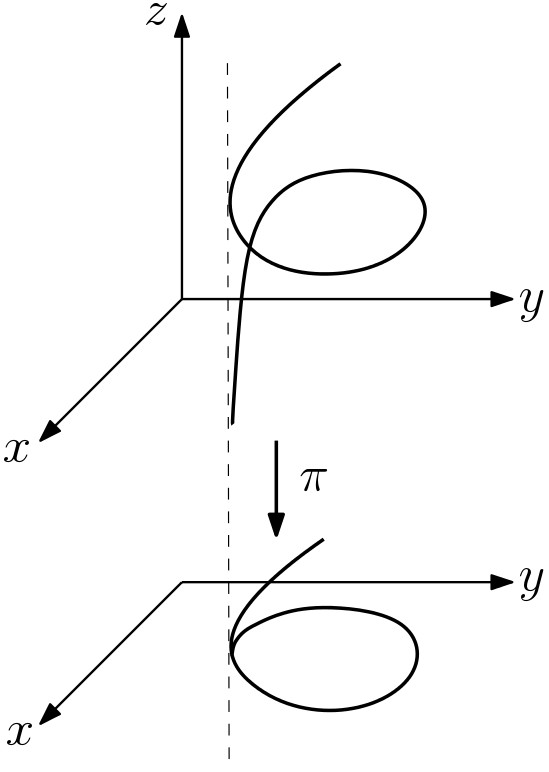}
            \subcaption{}
        \end{subfigure}%
        \begin{subfigure}{0.5\textwidth}
            \centering
            \includegraphics[height=0.6\textwidth]{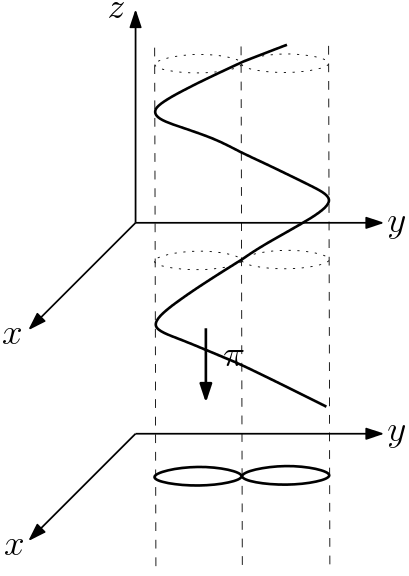}
            \subcaption{}
        \end{subfigure}
        \caption{Counterexamples}
        \label{counterexamples}
    \end{figure}\\
    It is also necessary to assume that each fiber of $f$ is connected. A counterexample shown in Figure \ref{counterexamples} (b) is constructed from the immersed manifold "figure $8$". By lifting it into $\mathbb{R}^3$ and smoothly extended the two ends, we get a smooth curve $\nu_2$ in $\mathbb{R}^3$ and $f$ is defined as before. Notice $f$ is a closed map but the preimage of certain point is not connected.
\end{remark*}
\begin{proof}[Proof of Lemma \ref{constantrankclosed}]
     We followed the suggestions posted in \cite{625526} to prove this lemma. The proof can be divided into three steps.

    \noindent Step 1. By \cite[C 3.3]{MR2304165}, the constant rank of $f$ shows $f(X)$ is locally a submanifold.
    That is, for every $x \in X$ and $y = f(x)$, there exists an open neighborhood $U_x \subset X$ of $x$ and an open neighborhood $V_y \subset Y$ of $y$ such that
    $f(U_x) \subset V_y$ is a submanifold of $V_y$. It suffices to prove that there is an open $W_y\subset V_y$ such that $f(X)\cap W_y=f(U_x) \cap W_y$, thus we have $f(X)$ is a submanifold of $Y$.
    \\
    Step 2. Since $f^{-1}(y)$ is connected, for any $x, x' \in f^{-1}(y)$, we can show $f$ maps the neighborhood of $x, x'$ to the "same" neighborhood of $f(y)$ by defining an equivalence relation. 
    More specifically, we say $x \sim x'$ if for any open neighborhood $O_x$ of $x$ and $O_{x'}$ of $x'$, there exists open neighborhood $U_x \subset O_x$ and $U_{x'} \subset O_{x'}$ such that $f(U_x) = f(U_{x'})$. 
    Observe that each equivalence class is an open set, and different classes are disjoint. 
    Therefore, if there are more than one equivalence class, then they form partition of $f^{-1}(y)$, which contradicts with the connectedness.
    \\
    Step 3. Back to what we want to prove, assume there is no such $W_y$ exists. 
    Then there exists a sequence ${y_k} \in f(X)$ converging to $y$ with each point distinct from $y$ but $y_k \notin f(U_x)$. Pick arbitrary preimage $x_k$ of $y_k$. 
    When $f$ is closed instead of proper, we have two cases. 
    If $\{x_k\}$ has limit points in $X$, then we can choose a subsequence $x_{k_j} \rightarrow x'$, which is forced to be in $f^{-1}(y)$.
    By Step 2, there should be $U_{x'}$ and $\tilde{U}_x \subset U_x$ such that $f(U_{x'}) = f(\tilde{U}_x)$. 
    For large enough $k$, we have $x_k \in U_{x'}$ and therefore $y_k \in f(U_{x'}) \subset f(U_x)$, which contradicts the assumption.
    If $\{x_k\}$ has no limit points, then $\{x_k\}$ is a closed subset $X$. The set $\{y_k\}=f(\{x_k\})$ has a limit point $y$ which is not contained in $\{y_k\}$. This contradicts with the assumption that $f$ is closed.
\end{proof}

\begin{cl}\label{waveset}
    If the microsupport  $\text{\( \WF \)}(A_1)$ and $\text{\( \WF \)}(A_2)$ are conically compact, then the wave front set  $\text{\( \WF \)}(A)$ is conically compact and is contained in $\Pi(\hat{W})$.
\end{cl}
\begin{proof}
    By \cite[Theorem 8.2.14]{Hoermander2003}, since $\text{\( \WF \)}(A_1), \text{\( \WF \)}(A_1)$ are away from the zero sections, then we have
    \[
    \text{\( \WF \)}'(A) \subset \text{\( \WF \)}'(A_1) \circ  \text{\( \WF \)}'(A_2),
    \]
    where $\text{\( \WF \)}'(\cdot)$ is the twisted relation. Hence, $\text{\( \WF \)}'(A) $ is a closed conic set contained in the image of the projection from the intersection $(\text{\( \WF \)}'(A_1) \times  \text{\( \WF \)}'(A_2)) \cap ( T^*X \times T^*\Delta(Y) \times T^*Z) $ to $T^*X \times T^*Z$. The intersection is conically compact so it is closed. Then the projection restricted there is continuous and it maps compact set to compact set. Moreover, it commutes with the multiplication by positive scalars in the covariant variables, and therefore the image of the intersection is conically compact. Thus, we have $\text{\( \WF \)}'(A) $  is conically compact.
\end{proof}
\begin{proof}[Proof of Claim \ref{Oclosed}]
    Since $ \hat{\mathcal{O}} \cap \hat{C}$ is an open set in $\hat{K}$, any closed subset $S$ is the intersection of $\hat{\mathcal{O}}$ with some closed subset $S_0$ of $\hat{K}$. 
    Notice $\hat{K}$ is conically compact and $\Pi$ preserves the fiber, 
    which implies $\Pi$ restricted to $\hat{K}$ is proper (therefore, is closed). It follows that $\Pi(S_0)$ is closed and therefore $\Pi(S) = \Pi(S_0) \cap \mathcal{O}$ is a closed subset of $\mathcal{O}$.
\end{proof}

\subsection*{Clean Composition}
A special case is when we compose the operator with its adjoint. As the following lemma shows, with certain condition the composition is clean. 
\begin{lm}\label{InjectiveImmersion}
    Suppose $C_2$ from $(T^*X \setminus 0)$ to $ (T^*Y \setminus 0)$ is a $C^\infty$ homogeneous canonical relation 
    closed in $ T^*(Y \times X)\setminus 0$. 
    Let $A_2 \in I^{m} (Y \times X, C_2)$ be a properly supported FIO associated with $C_2$. 
    If the projection $\pi_Y: C_2 \rightarrow T^*Y$ is an injective immersion, then the composition $A_2^*A_2$ is clean. In particular, the canonical relation of the composition map is the identity.
\end{lm}
We follow the same arguments in \cite{lightraytransform}.
\begin{proof}
    From \cite[Thm.25.2.2]{Hoermander2009}, we have $A_2^* \in I^{m} (X \times Y, (C_2^{-1})')$. 
    By definition, the composition is clean if $\hat{C_2} \equiv ( C_2^{-1} \times C_2 )\bigcap (T^*X \times \Delta (T^*X) \times T^*X)$ is a smooth manifold and its tangent space equals to the intersection of the tangent space of the intersecting manifolds.
    Indeed, let $\gamma_k = (x_k, \xi^k, y_k, \eta^k) \in C_2$, $k = 1, 2$ and $s: X \times Y \rightarrow Y \times X$ be the interchanging map. We have
    \[
    (s^* \gamma_1, \gamma_2) \in \hat{C} \Leftrightarrow \pi_Y(\gamma_1) = \pi_Y(\gamma_2).
    \]
    Since the projection is injective, then $\gamma_1 = \gamma_2$. This implies
    \[
    \hat{C} = \{(s^* \gamma, \gamma), \ \gamma \in C_2\}, 
    \quad T_{(s^*\gamma,\gamma)} \hat{C} = \{(s^*\gamma, s_* \delta_\gamma, \gamma, \delta_\gamma), \ \delta_\gamma \in T_\gamma C_2\}.
    \]
    On the other hand, we have  
    $(s_* \delta_{\gamma_1}, \delta_{\gamma_2}) \in T_{(s^* \gamma_1, \gamma_2)} (C_2^{-1} \times C_2)$ 
    is contained in the tangent space of $T^*X \times \Delta (T^*X) \times T^*X$, if and only if 
    \[
    \pi_Y(\gamma_1) = \pi_Y(\gamma_2), \quad \diff {\pi_Y}(\delta_{\gamma_1}) = \diff {\pi_Y}(\delta_{\gamma_2}).
    \]
    Since $\pi_Y$ is an injective immersion, 
    it follows that $\gamma_1 = \gamma_2$ and $\delta_{\gamma_1} =\delta_{\gamma_2}$, which proves the lemma.
\end{proof}
\begin{remark*}
In the setting of this lemma, the connectedness condition (A3) in Proposition \ref{CIC} is not needed. 
Observe that from the proof of Proposition \ref{CIC} and Lemma \ref{constantrankclosed}, the connectedness of $C_\gamma$ 
is required to guarantee that the composition $C_1 \circ C_2$ does not intersect itself. 
However, in this case we have the composition is the diagonal of $T^*X \times T^*X$, which is automatically not self-intersecting.
\end{remark*}
\section{The normal operator $I_\kappa^*I_\kappa$ as a $\Psi$DO}\label{sectionnormal}
In order to apply the clean composition theorem in Proposition \ref{CIC} to $I_\kappa^*I_\kappa$, 
we need the composition satisfying three assumptions (A1), (A2), (A3). 

For (A1), by Proposition \ref{propertyofC} and Lemma \ref{InjectiveImmersion}, if $T^*M$ is accessible, then the composition ${I_\kappa}^* I_\kappa$ is clean. 
For (A3), see the remark after Lemma \ref{InjectiveImmersion}. In most case in application, the surface $ \mathcal{S}$ is a plane, which makes the situation simpler.
As for (A2), we would like to show that with certain assumptions on the support of $\kappa$ (or by choosing proper smooth cutoff functions), 
we can find a compact subset $K$ of $C_I$ such that the microsupport of $I_\kappa$ (or multiplied by cutoffs) is contained in some open subset of $K$.

\begin{lm} \label{microsupportcutoff}
    Let $\chi_1(z), \chi_2(u, \phi)$ be smooth cutoff functions with compact supports. 
    Then $\chi_2 I_\kappa \chi_1$ is a Lagrangian distribution with conically compact microsupport support in $C_I$. 
    Additionally, there exists a compact set $K \subset C_I$ such that 
    $\text{\( \WF \)}(\chi_2 I_\kappa \chi_1)$ is contained in some open subset of $K$.
    In particular, these statements are true when $\kappa$ has compact support in $\mathcal{M} \times M$.
\end{lm}
\begin{proof}
    As is shown in the proof of Propostion \ref{propertyofC} (c), we have $(u,\beta,\hat{\phi},z)$ is a parameterization of $C_I$. 
    Thus, the map 
    \begin{align*}
    F_0 : & \mathcal{S} \times S^2 \times \mathbb{R} \setminus 0 \times M \rightarrow {C}_I\\
    & (u,\beta, \hat{\phi}, z) \mapsto (u,\beta,\phi,\hat{u}, \hat{\beta},\hat{\phi}, z, \zeta)
    \end{align*}
    is a continuous submersion. 
    Since we have continuous functions map compact sets to compact sets
    and submersions are open maps that map open sets to open sets, 
    this implies $F_0$ maps compact (or open) set in  $\mathcal{S} \times S^2 \times \mathbb{R} \setminus 0 \times M$  
    to compact (or open) set in $C_I$.
    
    In fact, we have $F_0$ maps conically compact (or open) set to conically compact (or open) set. 
    Indeed, suppose we have a conically compact neighborhood in $\mathcal{S} \times S^2 \times \mathbb{R} \setminus 0 \times M$, we can modify it to a compact neighborhood by restricting $|\hat{\phi}|$. Then the image of the compact set is compact in $C_I$ and therefore is compact in $C_I$ restricted to the cosphere bundle.
    Since $\hat{u}, \hat{\beta}, \hat{\phi}, \zeta$ is homogeneous of order 1 w.r.t. $|\hat{\phi}|$, we have the image of the conically compact (or open)  set is conically compact (or open).
    
    For the first statement of this lemma, notice that $\beta \in S^2$ is compact and we have the compact supports w.r.t $z,u$. Since $\phi \in (\epsilon, \pi/2 -\epsilon)$ might not be the whole range of $\phi = \cos^{-1}(\beta \cdot \frac{z-u}{|z-u|})$, we additionally assume we have compact support w.r.t $\phi$.
    To show the existence of the compact set $K$, observe that for any compact set $K'$ in $M, \mathcal{S}$ or $(\epsilon, \pi/2 -\epsilon)$, we can find a larger compact set such that $K'$ is in some open set of this larger compact set. 
    By the arguments above, it proves the second statement.
\end{proof}
 With the lemma above, assuming $\kappa$ has compact support, we can apply Proposition \ref{CIC} to $I^*_\kappa I_\kappa$ to show it is a $\Psi$DO. Moreover, with additional assumptions it is an elliptic $\Psi$DO, according to the formula \ref{principalsymbol} for composed principal symbols.
\begin{pp}\label{I8I}
    Assume $T^*M$ is accessible. Suppose the weight function $\kappa(u,\beta, \phi, z)$ has compact support in $\mathcal{M} \times  M$.
    If we have $\kappa(u_0,\beta_0,\phi_0,z_0) \neq 0$ for some $(u_0,\beta_0,\phi_0) \in \mathcal{C}(z_0,\zeta^0)$,
    then ${I_\kappa}^* I_\kappa$ is a $\Psi$DO of order $-2$ elliptic at $(z_0,\zeta^	0)$.
\end{pp}
\begin{remark*}
    The order is calculated by $(-3/2) + (-3/2) + e/2$, where the excess $e$ equals to the dimension of $C_\gamma$. 
    By Proposition \ref{propertyofC}, we have $C_\gamma \cong C_I(z,\zeta)$, of which the dimension is $2$.
\end{remark*} 

If it is not true that the whole cotangent bundle $T^*M$ is accessible, then we can have a microlocal version 
Proposition \ref{I8I} by choosing proper cutoff $\Psi$DOs.
More specifically, for a fixed accessible covector $(z_0,\zeta^0)$, 
there exists a conic neighborhood $\Gamma_0$ of $(z_0, \zeta^0)$ such that each covector in this neighborhood is accessible.
By choosing a cutoff $\Psi$DO which is supported in this neighborhood, then we can prove the following proposition.

\begin{pp}\label{mI8I}
    Suppose $(z_0,\zeta^0)$ is accessible. 
    Suppose $\kappa(u_0, \beta_0, \phi_0,z_0) \neq 0$ for some covector $(u_0,\beta_0,\phi_0,\hat{u}^0,\hat{\beta}^0,\hat{\phi}^0)$ 
    in $C_I(z_0,\zeta^0)$.
    Then there exists a $\Psi$DO $P(z,D_z)$ of order zero elliptic at $(z_0, \zeta^0)$ with microsupport in a conically compact neighborhood that is accessible,
    such that for any $Q = Q(u,\beta,\phi,D_u,D_\beta,D_\phi)$ that is a $\Psi$DO of order zero elliptic at $(u_0,\beta_0,\phi_0,\hat{u}^0,\hat{\beta}^0,\hat{\phi}^0)$ with conically compact support,
    the microlocalized normal operator $(Q{I_\kappa}P)^* Q I_\kappa P$ is a $\Psi$DO of order $-2$ elliptic at $(z_0,\zeta^0)$.
\end{pp}

Theorem 1 is the direct result of this proposition and the following example is a special case when we have unrecoverable singularities.
\begin{example}
    Let the surface of vertices $\mathcal{S} = \{(x^1, x^2, 0)|\  x^1 x^1 + x^2 x^2 < R^2\}$ be an open disk with radius $R$ in the $xy$ plane. Let $M$ be an open domain above $xy$ plane.
    
    Consider a covector $(z, \zeta) \in T^*M$ with $\zeta_1$ or $\zeta_2$ nonzero. Otherwise, it is not accessible.
    There is a cone $c(u,\beta,\phi)$ with vertex $u = (u^1, u^2, 0)$ in $\mathcal{S}$ conormal to $(z,\zeta)$ only if 
    \begin{multicols}{3}
        \noindent
        \begin{align*}
        &(z-u)\cdot \zeta = 0\\
        & u^1 u^1 + u^2 u^2 < R^2
        \end{align*}
        \begin{align*}
        \Rightarrow
        \end{align*}
        \begin{align*}
        &u^1 \zeta_1 + u^2 \zeta_2 = z \cdot \zeta \\
        & u^1 u^1 + u^2 u^2 < R^2
        \end{align*}
    \end{multicols}
    If we have $\zeta_2 \neq 0$, we can solve $u^2 = \frac{z \cdot \zeta - v^1 \zeta_1}{\zeta_2}$ from the first equation, and plug it in the second one to have
    \[
    u^1 u^1 +  \frac{(z \cdot \zeta - v^1 \zeta_1)^2}{(\zeta_2)^2} < R^2.
    \]
    For simplification, we denote $u^1$ by $t$ to get
    \[
    ((\zeta_1)^2+(\zeta_2)^2) t^2 - 2\zeta_1(z\cdot \zeta) t + (z \cdot \zeta)^2 - (\zeta_2)^2R^2 <0.
    \]
    This is a parabola opening to the top. There exists a solution of $t$ if and only if 
    \[
    \Delta = b^2 - 4ac =4\big((\zeta_1)^2(z\cdot \zeta)^2 - ((\zeta_1)^2+(\zeta_2)^2)((z \cdot \zeta)^2 - (\zeta_2)^2 R^2) \big) > 0,
    \]
    which implies
    \begin{equation}
    \label{condition_zeta}
    ((\zeta_1)^2+(\zeta_2)^2)R^2 - (z\cdot \zeta)^2 >0.
    \end{equation}
    If $\zeta_2 = 0$, then we must have $\zeta_1 \neq 0$. By symmetry we should get the same result. 
    Notice, if we can find such $u$ in $\mathcal{S}$ that $(z-u) \cdot \zeta =0$, then we can construct a cone $c(u,\beta, \phi)$ conormal to $(z,\zeta)$ by properly choosing $\beta$ and $\phi$. 
    Thus, the set of the unrecoverable singularities is
    \[
    \{(z,\zeta)|\  ((\zeta_1)^2+(\zeta_2)^2)R^2 - (z\cdot \zeta)^2 <0 \}.
    \]
    \begin{figure}[h] 
        \centering
        \includegraphics[height=0.4\textwidth]{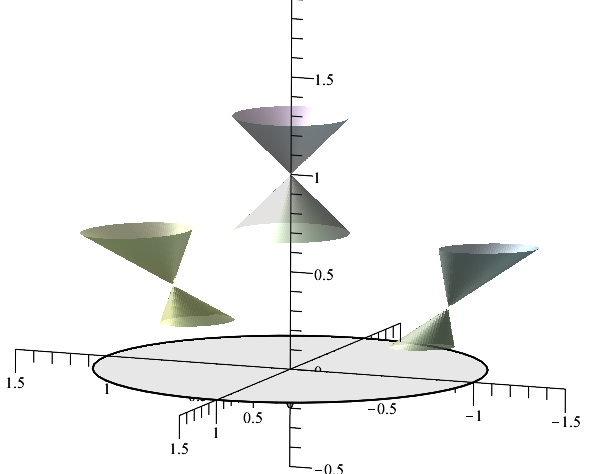}
        \caption{unrecoverable singularities at three points}
    \end{figure}
\end{example}
\section{proof of theorem 2}\label{sectionproof2}
For the $L^2$ estimates, we have the following restatement of \cite[Theorem 4.1.9 and Thmeorem 4.3.2]{Hoermander1791},
which indicates the mapping properties in the general case.
\begin{pp}\label{L2cts}
    Let $C$ be a homogeneous canonical relation from $T^*M$ to $T^*\mathcal{M}$ such that 
    \begin{itemize}
        \item[(1)] the maps $C \rightarrow M$ and $C \rightarrow \mathcal{M}$ have subjective differentials,
        \item[(2)] the projections $\pi_M: C \rightarrow T^*M$ and $\pi_\mathcal{M}: C \rightarrow T^*\mathcal{M}$ have constant rank.
    \end{itemize} 
    Let $k_M = \rank(\diff \pi_M) - \dim M$ and $k_{\mathcal{M}} = \rank(\diff \pi_\mathcal{M}) - \dim \mathcal{M}$ respectively. 
    Then $k_M = k_{\mathcal{M}} \equiv k$. 
    Suppose $A  \in I^m(\mathcal{M}\times M, C)$ is properly supported. Then 
    \[
    A: H^s_{loc}(M) \rightarrow H^{s+t}_{loc}(\mathcal{M}), \quad A^*: H^s_{loc}(\mathcal{M}) \rightarrow H^{s+t}_{loc}({M})
    \]
    are continuous,
    for any $t \leq (2k - \dim M - \dim \mathcal{M})/4 - m$ and $s \in \mathbb{R}$.
\end{pp}
\begin{proof}
    This corollary is a direct result of the theorems above. 
    Let $\Lambda^\sigma$ be the $\Psi$DO of order $\sigma$ such that 
    $\Lambda^\sigma u = \int e^{i x \cdot \xi} (1+|\xi|^2)^{\frac{\sigma}{2}}  \hat{u}(\xi) \diff \xi$.
    We apply \cite[Thm.4.3.2]{Hoermander1971} to $\tilde{A} = \Lambda^{s+t} A \Lambda^{-s} \in I^{m+t}(\mathcal{M}\times M, C') $. 
    Then $\tilde{A}: L^2_{loc}{M} \rightarrow L^2_{loc}(\mathcal{M})$ continuously, for any $m+t \leq (2k - \dim M - \dim \mathcal{M})/4$.
    Notice if $C$ satisfies the conditions (1)(2) in above corollary, then $C^{-1}$ satisfies them as well. 
    Since $A^* \in I^m(M \times \mathcal{M}, (C^{-1})')$, we can apply the same argument.
\end{proof}
Now we consider the conical Radon transform and its canonical relation $C_I$. We have the following claim by Proposition \ref{L2cts}.
\begin{cl*}
    When $T^*M$ is accessible, the canonical relation $C_I$ satisfies condition (1) and (2). In particular, we have $k = \dim M = 3$ and the inequality in Corollary \ref{L2cts} is $t\leq 1$.
\end{cl*}
\begin{proof}
    By Proposition \ref{propertyofC}, the projection $\pi_\mathcal{M}$ is an injective immersion, which implies $\pi_M$ is a submersion. Thus, $k = \dim(T^*M) - \dim M = \dim M$. 
    From the proof of Proposition \ref{propertyofC}, we have $C_I$ is parameterized by $(u,\beta,z,\hat{\phi})$. 
    It is obvious that $C \rightarrow M$ has subjective differential. 
    For the projection $C \rightarrow \mathcal{M}$, the Jacobian is in the following,
    \begin{figure}[h] 
        \centering
        \includegraphics[height=0.2\textwidth]{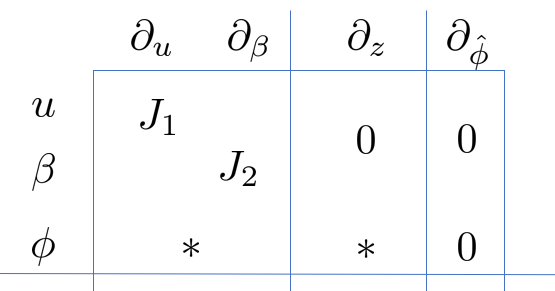}
    \end{figure}
    where $* = \frac{\partial \phi}{\partial z}$.
    Since $\phi$ is solved from $(z-u) \cdot \beta - |z-u| \cos \phi =0$, differentiating w.r.t. $z$ we have
    \[
    \frac{\partial \phi}{\partial z} = - \frac{1}{|z-u|\sin \phi}(\beta - \frac{z-u}{|z-u|}\cos \phi).
    \]
    Thus, the vector $\frac{\partial \phi}{\partial z}$ is nonzero and the differential of projection has rank equals $\dim \mathcal{M}$.
\end{proof}
By the claim above, we can prove Theorem 2 in the following based on similar arguments in \cite{MR2970707,Assylbekov2018}
\begin{proof}[Proof of Theorem 2]
    By Corollary \ref{L2cts} and the claim above, we have the second inequality and 
    $
    \|I^*_\kappa g\|_{H^{s+1}(M)} \leq C \|g\|_{H^s(\mathcal{M})},
    $
    for some constant $C$.
    From Proposition \ref{I8I}, we have 
    $
    C_1 \| f\|_{H^s(M)} - C_{s,l} \| f\|_{H^l(M)} \leq \|I^*_\kappa I_\kappa f\|_{H^{s+2}(M)}.
    $
    Combining these two we have desired result.
\end{proof}
\begin{proof}[Another proof of the estimate for $s = -1$]
    First we prove $I_\kappa: H_{\bar{M}}^s \rightarrow H^{s+1}(\mathcal{M})$ is bounded, for $ s \geq 0$, 
    where $H_{\bar{M}}^s$ is the space of all $u \in H^s(\mathbb{R}^3)$ supported in $\bar{M}$.
    Indeed, we have
    \[
    \| I_\kappa f \|^2_{H^{s+1}(\mathcal{M})} 
    \leq C\sum_{|\alpha| \leq 2s+2} |(\partial_{(u,\beta,\phi)}^\alpha I_\kappa f, I_\kappa f)_{L^2(M)}|
    = C\sum_{|\alpha| \leq 2s+2} |(I^*_\kappa \partial_{(u,\beta,\phi)}^\alpha I_\kappa f, f)_{L^2(M)}|.
    \]
    Here we have
    \[
    \partial_{(u,\beta,\phi)}^\alpha I_\kappa f = \sum_{\alpha_1+ \alpha_2 = \alpha}\int 
    \big(\partial_{(u,\beta,\phi)}^{\alpha_1} \kappa \big) \delta(g)
    +\kappa \big(\partial_{(u,\beta,\phi)}^{\alpha_2} g \big) \delta^{(|\alpha_2|)}(g) f(z) \diff z,
    \]
    where $\alpha, \alpha_1,\alpha_2$ are multi-indexes and $g = (z-u)\cdot \beta - |z-u|\cos\phi$.
    Since we always have $|\partial_z g| \neq 0$, locally we can write $\delta^{(1)}(g)$ as $\frac{1}{\partial_{z_j} g} \partial_{z_j} \delta(g)$, 
    where $j$ is the index such that $\partial_{z_j} g \neq 0$. Therefore by integration by parts and induction, 
    we get the similar integral transform of derivatives of $f$ up to order $|\alpha|$ with a new weight function. 
    This implies  $I^*_\kappa \partial_{(u,\beta,\phi)}^\alpha I_\kappa f$ is a $\Psi$DO of order $|\alpha| -2$. 
    Thus, we have the estimates
    \[
    \| I_\kappa f \|^2_{H^{s+1}(\mathcal{M})} \leq C \|f\|_{H^s(M)}.
    \]
    When $s=-1$, the proof is simplified. For non-integer $s \geq -1$, we can use interpolation methods.
    Then by duality, we have $I_\kappa^*: H^{-s-1}_{\bar{\mathcal{M}}} \rightarrow H^{-s}(M)$ is bounded, for $(H^{s+1}(\mathcal{M}))^* = H^{-s-1}_{\bar{\mathcal{M}}}$ and $(H^s(M))^* = H_{\bar{M}}^s$. 
    Since we always assume $\kappa(u,\beta,\phi,z)$ is compactly supported, we have $I_\kappa f$ has support in $\mathcal{M}$.
    Thus,
    \[
    \|I^*_\kappa I_\kappa f\|_{H^{-s}(M)} \leq C \|I_\kappa f\|_{H^{-s-1}(\mathcal{M})}.
    \]
    Combining these two inequality and the ellipticity of $I^*_\kappa I_\kappa$ , we have for $l<-1$,
    \[
    C_1\|f\|_{H^{-1}} - C_{l} \| f\|_{H^l(M)} \leq \|I_\kappa f\|_{L^2(\mathcal{M})} \leq C_2\|f\|_{H^{-1}}.
    \]
    We abuse the notation $C, C_1, C_2$ to denote different constants.
\end{proof}
\section{Restricted Cone Transform}\label{sectionrestricted}
In this section, suppose $\mathcal{S}_0 $ is a smooth regular curve that is parameterized by $u(t)$. For fixed $\phi_0 \in (\epsilon, 1-\epsilon)$, we define the restricted cone transform as
\[
I_{\kappa,\phi_0} f(u,\beta) = \int \kappa(u,\beta,z) \delta((z-u)\cdot \beta - |z-u|\cos \phi_0) f(z) \diff z.
\]
We have the following corollaries.
\begin{corollary}
    The restricted cone transform $I_{\kappa,\phi_0}$ is an FIO of order $-1$ associated with the canonical relation
    \begin{align*}
    C_{\phi_0}
    & = \{ (u, \beta, 
    \underbrace{ u'(t) \cdot  \zeta \vphantom{\lambda |z-u|\sin\phi}}_{\hat{u}},
    \underbrace{\lambda \transp{J_2}  (z-u) \vphantom{\lambda |z-u|\sin\phi}}_{\hat{\beta}},
    z,
    \zeta 
    ), \ 
    \varphi_{\phi_0}=0
    \},
    \end{align*}
    where 
    $\varphi_{\phi_0}(u,\beta,z) = (z-u) \cdot \beta - |z-u| \cos \phi_0$
    and $\zeta = -\lambda(\beta - \frac{z-u}{|z-u|}\cos\phi_0)$; 
    the vertex $u = u(t)$; 
    the unit vector $\beta$ is parameterized in the spherical coordinates 
    and $J_2 $ is the Jacobian matrix defined as before.
\end{corollary}
\begin{corollary}
    Let $D_0$ be the set of all $(z,\zeta)$ in $T^*M $ that are accessible w.r.t. $\mathcal{S}_0$. 
    Then $D_0$ is an open set.
    Let $R_0 =\{ (u,\beta,\phi,\hat{u},\hat{\beta},\hat{\phi}) \in C_I(z,\zeta)|(z,\zeta) \text{ is accessible}\}$.
    We have the following properties.
    \begin{itemize}
        \item[(a)] For every 
        $(u,\beta,\hat{u},\hat{\beta}) 
        \in R_0 $, 
        there is one unique solution $(z,\zeta)$ for the equation $C_{\phi_0}(z,\zeta)=(u,\beta,\hat{u},\hat{\beta})$, which is given by (\ref{solution_z0}) and (\ref{solution_z1}) .
        \item[(b)] The projection $\pi_\mathcal{M}$ restricted to $\pi^{-1}_\mathcal{M}(D_0)$ is an injective immersion. 
        In particular, if $\mathcal{S_0}$ satisfies Tuy's condition, then $\pi_\mathcal{M}$ itself is an injective immersion. 
    \end{itemize}
\end{corollary} 
\begin{proof}
    For (a), notice we have $\lambda|z-u| = \frac{\hat{\phi}}{\sin \phi} = \frac{1}{\sin \phi}|(\hat{\beta}_1, \frac{1}{\sin \theta} \hat{\beta}_2)|$. By (\ref{solution_z}) and its proof, when $\theta \neq 0$ or $\pi$,
    \begin{align}\label{solution_z0}
    &z = u + \frac{|(\hat{\beta}_1, \frac{1}{\sin \theta} \hat{\beta}_2)|}{\lambda \sin\phi} m, \quad \zeta = - \lambda(\beta - \cos\phi m )
    \end{align}
    where
    \[
    m=\cos \phi \beta + \frac{\sin \phi}{|(\hat{\beta}_1, \frac{1}{\sin \theta} \hat{\beta}_2)|}(\hat{\beta}_1 \beta_1 + \frac{1}{\sin \theta}\hat{\beta}_2 \beta_2).
    \]
    We still need to solve $\lambda$ from $(u,\beta,\hat{u},\hat{\beta})$. Note that $\hat{u} = u'(t) \cdot \zeta = -\lambda u'(t) \cdot (\beta - \frac{z-u}{|z-u|}\cos \phi_0) \neq 0$, since $(z,\zeta)$ is accessible. Therefore, we have
    \begin{align}\label{solution_z1}
    \lambda = -\frac{\hat{u}}{u'(t) \cdot (\beta - \frac{z-u}{|z-u|}\cos \phi_0)}.
    \end{align}
    For (b), to prove that $\pi_{\mathcal{M}}$ restricted to $\pi^{-1}_\mathcal{M}(D_0)$ is an immersion, 
    it suffices to show that $C_{\phi_0}$ is parameterized by $(u,\beta,\hat{u},\hat{\beta})$.
    Indeed, from (b), $(z,\zeta)$ can be represented by $(u,\beta,\hat{u},\hat{\beta})$. 
    By writing $z = u + \rho \cos \phi_0  \beta + \rho \sin \phi_0 (\sin \alpha \beta_1 + \cos \alpha \beta_2)$ and performing a same argument as in the proof of Proposition \ref{InjectiveImmersion}, we can show it is a parameterization 
    and therefore the rank of the differential of $\pi_{\mathcal{M}}$ equals $\dim C_{\phi_0} = 6$.
\end{proof}
\begin{remark*}
Based on similar arguments, we can also show the same results hold if we only restrict the vertexes on a smooth curve $\mathcal{S}_0$ without $\phi$ fixed. Additionally, one can show an analog of recovery of singularities in both cases as in Theorem 1. 
\end{remark*}
\begin{footnotesize}
    \bibliographystyle{plain}
    \bibliography{microlocal_analysis}
\end{footnotesize}
\end{document}